\documentclass[12pt]{article}
\usepackage{a4wide}
\usepackage{amsmath,amssymb,amsthm}
\usepackage[mathscr]{eucal}
\usepackage[shortlabels]{enumitem}
\usepackage{hyperref}
\usepackage{eqparbox}

\author{Ievgen Bondarenko}
\title{\textbf{The word problem and growth of groups}}

\newcommand{\DTIME}{\mathrm{DTIME}}
\newcommand{\WP}{\mathrm{WP}}
\newcommand{\Poly}{\mathrm{P}}
\newcommand{\CFL}{\mathrm{CFL}}

\newcommand{\boxchar}[2][A]{\fbox{\eqmakebox[#1][c]{$\mathstrut#2$}}}

\newtheorem{theorem}{Theorem}
\newtheorem{proposition}[theorem]{Proposition}
\newtheorem{corollary}{Corollary}[theorem]

\theoremstyle{definition}
\newtheorem{definition}{Definition}
\newtheorem{example}{Example}

\newtheorem{question}{Question}

\begin{document}
\maketitle

\begin{abstract}
Let $\WP_G$ denote the word problem in a finitely generated group $G$. We consider the complexity of $\WP_G$ with respect to standard deterministic Turing machines. Let $\DTIME_k(t(n))$ be the complexity class of languages solved in time $O(t(n))$ by a Turing machine with $k$ tapes. We prove that $\WP_G\in\DTIME_1(n\log n)$ if and only if $G$ is virtually nilpotent. We relate the complexity of the word problem and the growth of groups by showing that $\WP_G\not\in \DTIME_1(o(n\log\gamma(n)))$, where $\gamma(n)$ is the growth function of $G$. We prove that $\WP_G\in\DTIME_k(n)$ for strongly contracting automaton groups, $\WP_G\in\DTIME_k(n\log n)$ for groups generated by bounded automata, and $\WP_G\in\DTIME_k(n(\log n)^d)$ for groups generated by polynomial automata. In particular, for the Grigorchuk group, $\WP_G\not\in\DTIME_1(n^{1.7674})$ and $\WP_G\in\DTIME_1(n^2)$.

\vspace{0.2cm}\textit{2020 Mathematics Subject Classification}: 20F10, 68Q70, 03D10, 20E08

\end{abstract}


\section{Introduction}

Let $G$ be a finitely generated group, and $S$ be a finite generating set closed under inversion. The word problem in $G$ with respect to $S$ is the language $\WP(G,S)=\{ w\in S^{*} : w=_Ge \}$. Different properties of the language $\WP(G,S)$ are typically do not depend on the choice of $S$, and one can talk about groups whose word problem $\WP_G$ satisfies a given property.

The study of the word problem is rich with many beautiful results. In this paper, we consider the word problem with respect to the time complexity of standard deterministic Turing machines. For our purpose, we will distinguish Turing machines with a single tape and multiple tapes. Let $\DTIME_1(t(n))$ and $\DTIME_{*}(t(n))$ be the complexity classes of all languages solved in time $O(t(n))$ by a Turing machine with a single tape and multiple tapes respectively. Note that $\DTIME_{*}(t(n))\subseteq\DTIME_1(t(n)^2)$ by a classical result of Hartmanis and Stearns \cite{Hartmanis}.

Anisimov \cite{Anisimov} proved that the language $\WP_G$ is regular if and only if the group $G$ is finite. Note that the class of regular languages coincides with the complexity class $\DTIME_1(n)$. In contrast, the class $\DTIME_{*}(n)$ contains the word problems for many classes of groups: hyperbolic groups, finitely generated nilpotent and geometrically finite hyperbolic groups (see \cite{Holt:realtime}); there is no known description of groups in $\DTIME_{*}(n)$. The classical result of Kobayashi \cite{Kobayashi} states that there is a gap in the time complexity between $O(n)$ and $O(n\log n)$ for single-tape Turing machines, namely $\DTIME_1(n)=\DTIME_1(o(n\log n))$. In particular, $\WP_G\in\DTIME_1(o(n\log n))$ implies the group $G$ is finite. The word problem in the group $\mathbb{Z}$ separates the classes $\DTIME_1(n)$ and $\DTIME_1(n\log n)$. We characterize the word problem in the class $\DTIME_1(n\log n)$.

\begin{theorem}
Let $G$ be a finitely generated group. Then $\WP_G\in \DTIME_1(n\log n)$ if and only if $G$ is virtually nilpotent.
\end{theorem}

In one direction, we use the existence of expanding endomorphisms of the unitriangular group $UT_n(\mathbb{Z})$, proved in \cite{GoodmanShapiro}. For the opposite, using the standard technique from the complexity theory, we relate the growth of a group and the complexity of its word problem.

\begin{theorem}
Let $G$ be a finitely generated group with the growth function $\gamma_G(n)$. Then $\WP_G\not\in \DTIME_1(o(n\log \gamma(n)))$. In particular, for groups of exponential growth the word problem in not solvable in time $o(n^2)$ by a single-tape Turing machine.
\end{theorem}

Then Gromov's celebrated result on groups of polynomial growth implies that for non-virtually nilpotent groups the word problem is not solvable in time $O(n\log n)$ by a single-tape Turing machine. Moreover, the known gap in the growth of groups between polynomial and $n^{(\log\log n)^c}$ for certain $c>0$ proved in \cite{ShalomTao} implies the gap in the complexity of the word problem: there is no word problem strictly between the complexity classes $\DTIME_1(n \log n )$ and $\DTIME_1(n\log n(\log\log n)^c)$. The Grigorchuk's gap conjecture (that $\gamma_G(n)\prec exp(\sqrt{n})$ implies polynomial growth, see \cite{Gri:growth} and \cite{Gri:Milnor}) entails the gap for the word problems between $\DTIME_1(n\log n)$ and $\DTIME_1(n^{3/2})$. I actually do not know examples of groups with the word problem in $\DTIME_1(o(n^2))\setminus \DTIME_1(n\log n)$.

\begin{question}
Does $\WP_G\in\DTIME_1(o(n^2))$ imply $G$ is virtually nilpotent?
\end{question}

The question is specific to groups of intermediate growth. A large known class of such groups appears as automaton groups, i.e., groups generated by Mealy automata. The word problem in automaton groups is solvable in exponential time and polynomial space, and there exists an automaton group with a $\textsf{PSPACE}$-complete word problem (see \cite{WachterWeiss:PSPACE}). The well-known examples of automaton groups of intermediate growth are generated by bounded automata. Bounded automaton groups belong to the class of contracting groups, where the word problem is solvable in polynomial time. We prove that linearithmic time is sufficient for multi-tape Turing machines.

\begin{theorem}
Let $G$ be a group generated by a bounded automaton. Then $\WP_G\in \DTIME_{*}(n\log n)$. If $G$ is strongly contracting, then $\WP_G\in \DTIME_{*}(n)$.
\end{theorem}

All strongly contracting groups have subexponential growth. The famous example is the Grigorchuk group $G$ (see \cite{Grigorchuk:Degrees}). The growth function of $G$ is essentially equivalent to $\exp(n^\alpha)$, where $\alpha=\frac{\log 2}{\log \eta}\approx 0.7674$ and $\eta$ is the positive root of the polynomial $x^3-x^2-2x-4$ (see \cite{ErschlerZheng:GriGrowth}). Since $\DTIME_{*}(n)\subseteq\DTIME_1(n^2)$, we get the following corollary:

\begin{corollary}
Let $G$ be the Grigorchuk group. Then $\WP_G\in\DTIME_1(n^2)$ and $\WP_G\not\in\DTIME_1( o(n^{1+\alpha}) )$.
\end{corollary}

I do not know the exact complexity of the word problem in the Grigorchuk group. Since $G$ is a branch group (it contains a subgroup of finite index $K$ such that $K\times K$ is a subgroup of finite index in $K$), the word problem in $G$ could be not solvable in time $o(n^2)$ by a single-tape Turing machine.

Bounded automata are polynomial automata of zero degree. In \cite{Bondarenko:Schreier}, it is proved that for groups generated by polynomial automata the word problem is solvable in subexponential time. A more careful analysis gives quasilinear time.

\begin{theorem}
Let $G$ be a group generated by a polynomial automaton of degree $d$. Then $\WP_G\in \DTIME_{*}(n (\log n)^{(d+1)^2})$. If $G$ is contacting, then $\WP_G\in\DTIME_{*}(n(\log n)^{d+1})$.
\end{theorem}

\section{Time complexity and the word problem}

We assume the reader is familiar with the standard notions of single-tape and multi-tapes Turing machines, their computation, and time complexity (see, for example, the textbook \cite{sipser13}).

Let $X$ be an alphabet, whose elements are called letters. The set of all words (strings) over $X$ (including the empty word $\varepsilon$) is denoted $X^{*}$. With the operation of concatenation, the set $X^{*}$ is a monoid with the identity $\varepsilon$. A language over $X$ is any subset of $X^{*}$. For a word $w\in X^{*}$, we write $|w|$ for the length of $w$, the number of letters in $w$. Let $X^n$ denote the set of words of length $n$.

Let $G$ be a finitely generated group and $S$ be a finite generating set. We will always assume that $S$ is closed under inversion, that is, $a^{-1}\in S$ for every $a\in S$. Then every element $g\in G$ can be represented by a word $w\in S^{*}$. The \textit{word problem} in $G$ with respect to $S$ is the language $\WP(G,S)=\{w\in S^{*} : w=_Ge\}$, where $e$ is the identity of $G$.

Let $C$ be a class of languages. If the property of $\WP(G,S)$ being in $C$ is independent on the choice of finite generating set $S$ of $G$, we can talk about $\WP_G$ or just $G$ being in $C$. The standard condition that implies such a property of $C$ is when $C$ is closed under inverse homomorphisms. If $C$ is closed under intersection with regular languages and inverse GSMs, then $C$ is closed under passing to finitely generated subgroups and finite index overgroups (see \cite{Holt:GroupLangAut}). Familiar classes of languages satisfy these conditions. The complexity classes $\DTIME_{*}(t(n))$ for $t(n)\geq n$ and $\DTIME_1(t(n))$ for $t(n)=\Omega(n^2)$ are closed with respect to all mentioned properties. However, I do not know if the classes $\DTIME_1(t(n))$ for $t(n)=o(n^2)$ are closed under inverse GSMs or even inverse homomorphisms. The reason is that replacing letters with words on a single tape requires shifting the tape, which could take $\Omega(n^2)$ time in the worst case. (Actually, it seems that $\DTIME_1(n\log n)$ is not closed under inverse homomorphisms.) Nevertheless, we can use the next statement.

\begin{proposition}\label{prop:time_WP}
Let $C$ be the class $\DTIME_1(t(n))$ for the function $t(n)\geq n$ such that $t(Cn)=O(t(n))$ for all $C>0$.
\begin{enumerate}
  \item If $S$ and $T$ are two finite generating sets of a group $G$, then $\WP(G,S)\in C$ if and only if $\WP(G,T)\in C$.
  \item If $H$ is a finitely generated subgroup of $G$ and $\WP_G\in C$, then $\WP_H\in C$.
  \item If $H$ is a subgroup of finite index in $G$, then $\WP_H\in C$ if and only if $\WP_G\in C$.
\end{enumerate}
\end{proposition}
\begin{proof}
The proof basically repeats Propositions 3.4.3, 3.4.5, 3.4.7 in \cite{Holt:GroupLangAut} but in terms of Turing machines with a single tape. We write details only for item 1. Let $M$ be a Turing machine solving $\WP(G,S)$ in time $t(n)$. For every $a\in T$, fix a nonempty word $\tau(a)\in S^{*}$ such that $a=_G\tau(a)$. We get a monoid homomorphism $\tau:T^{*}\rightarrow S^{*}$, $\tau(a_1\ldots a_n)=\tau(a_1)\ldots\tau(a_n)$, so that $w\in\WP(G,T)$ if and only if $\tau(w)\in\WP(G,S)$. One can construct a Turing machine $M'$ that, given a word $w\in T^{*}$, simulates $M$ on $\tau(w)$ without calculating $\tau(w)$. The tape alphabet for the machine $M'$ consists of words $w\in S^{*}$ of length $\leq k$, where $k$ is the maximal length of $\tau(a)$ for $a\in T$, together with a possible mark on one of the letters indicating the current head position. Each input letter $a\in T$ is interpreted as the working symbol $\tau(a)$. The calculation of $M'$ on a symbol $v$ simulates the calculation $M(v)$. Each step of $M'(w)$ corresponds to a step of $M(\tau(w))$. Therefore, the running time of $M'$ is bounded by $t(kn)=O(t(n))$.

The item 2 follows from item 1 together with the fact that the class $C$ is closed under intersection and union with regular languages.

The item 3 repeats item 1 with $\tau$ replaced with the generalised sequential machine that, given a word over generators of $G$, computes its coset representative and a word over generators of $H$.
\end{proof}

\section{The word problem and growth of groups}

Let $G$ be a group generated by a finite set $S$ closed under inversion. The \textit{word length} $l_S(g)$ of an element $g\in G$ is the length of the shortest word $w\in S^{*}$ such that $w=_Gg$. The \textit{growth function} $\gamma_S(n)$ is equal to the number of elements $g\in G$ such that $l_S(g)\leq n$. If $T$ is another finite generating set of $G$, then $\gamma_T(n)\leq \gamma_S(kn)\leq \gamma_S(n)^k$ for some $k\geq 1$ and $\log \gamma_T(n)=O( \log \gamma_S(n) )$.

\begin{theorem}
Let $G$ be a finitely generated group with the growth function $\gamma(n)$. Then $\WP_G\not\in \DTIME_1(o(n\log\gamma(n)))$.
\end{theorem}
\begin{proof}
Let $M$ be a single-tape Turing machine that solves $\WP(G,S)$ in time $T(n)$. We will show that $T(n)=\Omega(n\log\gamma(n))$.

We use the standard method of crossing sequences. The \textit{crossing sequence} $C_i(w)$ on input $w$ at position $i$ is the sequence of states that $M$ is in when its head crosses the boundary between the $i$th and $(i+1)$th tape cells (in either direction) during the computation $M(w)$. Note that the total number of crossing sequences of length at most $t$ is bounded by $q^{t+1}$, where $q$ is the number of $M$ states.

Let $B_n\subset S^{*}$ be the set of words of even length at most $n$ that represent different non-trivial group elements. Note that $\frac{1}{2}\gamma(n)\leq |B_n|\leq \gamma(n)$. Fix $a\in S$ and put
\[
L_n=\{w(aa^{-1})^kw^{-1}: w\in B_{n} \mbox{ and $k=2n-|w|$}  \}\subset\WP(G,S).
\]
Note that $|L_n|=|B_n|$ and the words in $L_n$ have length $4n$.

Claim: If $x,y\in L_n$ and $x\neq y$, then $C_i(x)\neq C_j(y)$ for any even $i,j$ with $n<i,j<3n$. Indeed, suppose $C_i(x)=C_j(y)$ for some even $n<i,j<3n$. Let $x'$ be the prefix of $x$ of length $i$, and let $y'$ be the suffix of $y$ of length $4n-j$. Then the computation of $M$ on input $x'y'$ coincides with the computation $M(x)$ on the left side of the tape from the position $i$ and with the computation $M(y)$ on the right side of the tape from the position $j$ (see more details in \cite[Theorem 1]{Hennie}). Therefore, the computation of $M(x'y')$ accepts: the accepting step of either $M(x)$ or $M(y)$ has a corresponding step of $M(x'y')$. However, $x'=w_1(aa^{-1})^{k_1}$ and $y'=(aa^{-1})^{k_2}w_2^{-1}$ for different $w_1,w_2\in B_n$, and therefore $x'y'=_Gw_1w_2^{-1}\neq_G e$. Contradiction.

Each element of each crossing sequence corresponds to a move of $M$. Since the running time of $M(x)$ for each $x\in L_n$ is at most $T(4n)$, there should exist an even position $i$ (depending on $x$) between $n$ and $3n$ such that the crossing sequence $C_i(x)$ has length at most $\frac{1}{n}T(4n)$. These crossing sequences are different for different $x\in L_n$ by our claim. Therefore,\vspace{-0.5cm}
\begin{align*}
\frac{1}{2}\gamma(n) \leq |B_n|\leq q^{\frac{1}{n}T(4n)+1} \quad \Rightarrow \quad T(4n)\geq n\log \gamma(n).
\end{align*}
\end{proof}

\begin{corollary}
If $G$ is a finitely generated group of exponential growth, then $\WP_G\not\in \DTIME_1(o(n^2))$.
\end{corollary}

\begin{corollary}
If $G$ is a finitely generated group that is not virtually nilpotent, then $\WP_G\not\in \DTIME_1(n\log n)$.
\end{corollary}
\begin{proof}
By the Gromov's celebrated theorem, finitely generated virtually nilpotent groups are exactly groups of polynomial growth. Moreover, if the growth function $\gamma(n)$ is not bounded by a polynomial, then $\log n =o(\log \gamma(n))$, and we can apply the previous theorem.
\end{proof}

\begin{corollary}
Let $G$ be a finitely generated group with the growth function $\gamma_G(n)\succcurlyeq 2^{n^{\alpha}}$ for $\alpha\in(0,1]$. Then $\WP_G\not\in \DTIME_1(o(n^{1+\alpha}))$.
\end{corollary}

\section{The word problem in virtually nilpotent groups}

In this section, we prove that the word problem in a finitely generated virtually nilpotent group $G$ is solvable in time $O(n\log n)$ by a single-tape Turing machine. The idea is analogous to the solution of the word problem in the group $\mathbb{Z}$: check that the input word $w$ represents an even number $n$, compute a word for $n/2$ in linear time by  crossing every second appearance of a generator, repeat. The role of division by two will be played by the inverse to an expanding endomorphism.

\begin{theorem}
Let $G$ be a finitely generated virtually nilpotent group. Then $\WP_G\in \DTIME_1(n\log n)$.
\end{theorem}
\begin{proof}
Every finitely generated virtually nilpotent group contains a torsion-free nilpotent subgroup of finite index. Every finitely generated torsion-free nilpotent group embeds into a finitely generated nilpotent group that admits an expanding endomorphism (see \cite[Section~4]{GoodmanShapiro}). Therefore, in view of Proposition~\ref{prop:time_WP}, we may assume that $G=\langle S\rangle$ is a finitely generated group that admits an expanding endomorphism $\phi$, that is, $H=\phi(G)$ has finite index in $G$ and there exists a constant $C>1$ such that $l_S(\phi(g))\geq Cl_S(g)$ for all $g\in G$. By taking a suitable power of $\phi$, we may assume that $C\geq 4$. Note that $\phi:G\rightarrow H$ is an isomorphism, and we can talk about its inverse $\phi^{-1}:H\rightarrow G$.


Let $X$ be a set of right coset representatives for $H$ in $G$ containing the identity $e$. Let $R$ be the largest length $l_S(x)$ for $x\in X$. Let $N\subset G$ consist of elements with $l_S(g)\leq R$. Then $N$ is a finite generating set of $G$ closed under inversion that has the following property: for every $a,b\in N$ and $x\in X$ there exists a unique $y\in X$ such that $xaby^{-1}\in H$ and $\phi^{-1}(xaby^{-1})=c\in N$, where the last equation follows from the inequalities
\[
l_S(\phi^{-1}(xaby^{-1}))\leq \tfrac{1}{C} l_S(xaby^{-1})\leq \tfrac{1}{C} 4R\leq R.
\]
We can view this property as the rewriting rule $(a,b,x)\mapsto(e,c,y)$ that replaces a pair of letters $ab$ by the pair $ec$ (here $x,y$ could be stored in memory/states).

Given a word $w\in N^{*}$, we have $w=_Ge$ if and only if $w$ represents an element of $H$ and $\phi^{-1}(w)=_Ge$. We check whether $w\in H$, compute $w'\in N^{*}$ such that $w'=_G\phi^{-1}(w)$ by using our rewriting rules, and repeat the process with $w$ replaced by $w'$. Each iteration reduces the number of non-identity letters in $w$ in half. A Turing machine $M$ with a single tape over the alphabet $N$ implementing this process could be realized as follows:
\begin{enumerate}
  \item If the input tape word $w\in N^{*}$ contains only the letters $e$, accept.
  \item Scan the input tape word $w$ and compute its coset representative $x\in X$ in $H$ using the states of $M$. If $x\neq e$, reject. Otherwise, return to the beginning of the tape.
  \item Scan the input tape word $w$ skipping the letter $e$. At each step, the state is labeled by a letter $x\in X$ representing the current coset representative (initially $x=e$). For a pair of two consecutive symbols $a,b\in N\setminus\{e\}$, use the rule $(a,b,x)\mapsto (e,c,y)$: replace $a$ by $e$, replace $b$ by $c$, and move to the state of $M$ labeled with $y\in X$ (the coset representative of $xab$). We get a word $w'\in N^{*}$ of length $|w'|=|w|$ and $w'=_G\phi^{-1}(w)$.
\begin{align*}
w&=\boxchar{a_1}\boxchar{b_1}\boxchar{a_2}\boxchar{b_2}\boxchar{\ldots} \boxchar{a_k}\boxchar{b_k} &&=_G (a_1b_1y_1^{-1})(y_1a_2b_2y_2^{-1}) y_2\ldots y_{k-1}a_kb_k\\
w'&=\boxchar{e}\boxchar{c_1}\boxchar{e}\boxchar{c_2}\boxchar{\ldots} \boxchar{e}\boxchar{c_k} &&=_G \phi^{-1}(a_1b_1y_1^{-1})\phi^{-1}(y_1a_2b_2y_2^{-1})\ldots \phi^{-1}(y_{k-1}a_kb_k)
\end{align*}

  \item Return to the beginning of the tape and go to item 1.
\end{enumerate}
The items $1-4$ are performed in $O(n)$ steps and reduce the number of non-identity letters in an input word in half. Therefore, the computation finishes after $O(\log n)$ stages, and the running time of $M$ is $O(n\log n)$.
\end{proof}

By intersecting $\DTIME_1(n\log n)$ with context-free languages, we can separate virtually abelian groups. Let $k{\text -}\CFL$ be the class of languages that are intersections of exactly $k$ context-free languages.

\begin{corollary}
Let $G$ be a finitely generated group. Then $\WP_G\in \DTIME_1(n\log n)\cap k{\text -}\CFL$ if and only if $G$ is virtually abelian of rank at most $k$.
\end{corollary}
\begin{proof}
The result follows from Proposition 4.3 and Lemma 4.6 in \cite{Brough:PolyCF}: for a finitely generated nilpotent group,  $\WP_G\in k{\text -}\CFL$ if and only if $G$ is virtually $\mathbb{Z}^m$ for $m\leq k$.
\end{proof}

\begin{question}
Can the nilpotency class of a nilpotent group be determined by the complexity of the word problem?
\end{question}

\section{The word problem in automaton groups}

\subsection{Automaton groups}

Let us review necessary information about automaton groups (see \cite{Nekr:SSGroups} for more details).

Let $X$ be a finite alphabet, fixed for the rest of the paper. Automaton groups are generated by a special type of finite-state transducers that will be called just automata. An \textit{automaton} over $X$ is the tuple $A=(S,X,t,o)$, where $S$ is a finite set of states, $t:S\times X\rightarrow S$ the transition map, and $o:S\times X\rightarrow X$ the output map. The maps $t$ and $o$ are naturally extended to the maps $t:S\times X^{*}\rightarrow S$ and $o:S\times X^{*}\rightarrow X^{*}$ by the rules:
\begin{align*}
t(s,\epsilon)=s, \ t(s,xv)=t(t(s,x),v), \quad  o(s,\epsilon)=\epsilon, \ o(s,xv)=o(s,x)o(t(s,x),v),
\end{align*}
where $x\in X, v\in X^{*}$ and $s\in S$. The rules correspond to the standard interpretation of transducers as computational machines: the word $o(s,v)$ is the output of the computation $A(v)$ starting with the initial state $s$, while the state $t(s,v)$ is the terminal state after the computation.

An automaton $A$ is called \textit{invertible} if $o(s,\cdot)$ induces a permutation on $X$ for every $s\in S$. In this case, $o(s,\cdot):X^{*}\rightarrow X^{*}$ is a bijection for every $s\in S$, and the inverse transformations are also given by an automaton, the inverse to $A$. The \textit{automaton group} $G_A$ is the subgroup of $Sym(X^{*})$ generated by the transformations $o(s,\cdot)$, $s\in S$. Since we are interested only in groups, we always assume that the generating automaton is minimal, that is, different states define different transformations. Therefore, we can identify $s$ and $o(s,\cdot)$, and say that the group $G_A$ is generated by $S$. Further, speaking about automaton groups, we always assume that they are defined over the alphabet $X$ and $S$ is its automaton generating set closed under inversion, if it is not stated otherwise. Also, the maps $t$ and $o$ preserve the set $X^k$ for every $k\in\mathbb{N}$, so that we can talk about $A$ as an automaton over the alphabet $X^k$. Note that passing from $X$ to $X^k$ does not change the automaton group.

The representation of elements of $G_A$ by words in $S^{*}$ naturally comes from automata composition as follows.
The maps $t$ and $o$ are further extended to the maps $t:S^{*}\times X^{*} \rightarrow S^{*}$ and $o:S^{*}\times X^{*}\rightarrow X^{*}$ by the rules:
\begin{align*}
t(sw,v)=t(s,v)t(w,o(s,v)), \quad  o(\epsilon,v)=v, \ o(sw,v)=o(w,o(s,v)),
\end{align*}
where $x\in X, v\in X^{*}$ and $s\in S, w\in S^{*}$. The rules correspond to the right composition of automata. (We are using right actions, because the Turing machines that will be used for solving the word problem process words from left to right.) Then a word $w\in S^{*}$ represents the element $g=o(w,\cdot)\in G_A$.

The action of the group $G_A$ on the set $X^{*}$ is self-similar in the following sense. For $w\in S^{*}$ and $x\in X^{*}$, the word $t(w,x)\in S^{*}$ is called the \textit{section} of $w$ at $x$ and is denoted by $w|_x$. The words $w$ and $w|_x$ have the same length for every $x\in X^{*}$. Let $\pi_w$ denote the permutation of $X$ induced by $o(w,\cdot)$. Now, if the words $w$ and $w|_x$ represent group elements $g$ and $h$ respectively, then $g(xv)=\pi_w(x)h(v)$ for all $v\in X^{*}$.

Every automaton group $G_A$ has a solvable word problem. Indeed, a word $w\in S^{*}$ represents the identity element of $G_A$ if and only if $\pi_v=\varepsilon$ for every section $v$ of $w$, where $\varepsilon$ is the identity permutation. The sections $v$ and permutations $\pi_v$ are computable in linear time as follows. First, there exists a finite-state automaton over $S$ that, given a word $w\in S^{*}$, recognizes $\pi_w=\varepsilon$. The states are the permutations of $X$, the initial and final state is $\varepsilon$, and the arrows are $\pi\xrightarrow{s} \pi\cdot \pi_s$ for $s\in S$, $\pi\in Sym(X)$. Second, the dual to the automaton $A$ computes sections. Namely, consider the automaton $B=(X,S,t',o')$, where $t'(x,s)=o(s,x)$ and $o(x,s)=t(s,x)$. Then, given a word $w\in S^{*}$ and $x\in X$, the output of the computation $B(w)$ starting from the state $x\in X$ is exactly the word $w|_x$. Explicitly, we have
\begin{equation*}\label{eqn:section_word}
(s_1s_2\ldots s_n)|_x=s'_1s'_2\ldots s'_n, \ \mbox{where } s'_i=s_i|_{x_i} \mbox{ and } x_{i+1}=\pi_{s_i}(x_{i}), x_1=x.
\end{equation*}
Therefore, since the number of sections is at most exponential, the word problem in automaton groups is solvable in at most exponential time.

\subsection{The word problem in contracting automaton groups}

A more effective way of calculating words that represent sections could lead to a more effective way to solve the word problem.

\begin{definition}
An automaton group $G$ is called \textit{contracting} if $l_S(w|_x)<|w|$ for all $x\in X$ and all sufficiently long words $w\in S^{*}$.

An automaton group $G$ is called \textit{strongly contracting} if $\sum_{x\in X} l_S(w|_x)<|w|$ for all sufficiently long words $w\in S^{*}$.
\end{definition}

Note that a group may be (strongly) contracting for one generating automaton and not (strongly) contracting for another one. One can define a group $G$ to be \textit{(strongly) contracting} if it is (strongly) contracting for some automaton representation.

The word problem in contracting groups is solvable in polynomial time (see \cite[Proposition~2.13.10]{Nekr:SSGroups}).
This result can be deduced from the following analog of Master theorem in algorithm analysis.

\begin{theorem}\label{thm:timecomp_autom}
Let $G$ be an automaton group. Let $L\in\mathbb{N}$.
\begin{enumerate}
  \item If $l_S( w|_x )<|w|$ for all $w\in S^{L}$ and $x\in X$, then $\WP_G\in\Poly$. 
  \item If $\sum_{x\in X} l_S( w|_x )\leq |w|$ for all $w\in S^{L}$, then $\WP_G\in\DTIME_{*}(n\log n)$.
  \item If $\sum_{x\in X} l_S( w|_x )< |w|$ for all $w\in S^{L}$, then $\WP_G\in\DTIME_{*}(n)$.
\end{enumerate}
\end{theorem}
\begin{proof}
The condition in item 1 implies that there exists a constant $0\leq\lambda'<1$ such that $l_S(w|_x)\leq \lambda'|w|$ for all $w\in S^{L}$. For every $w\in S^{L}$ and $x\in X$, fix a word $w_x\in S^{*}$ such that $w_x=_Gw|_x$ and $|w_x|=l_S(w|_x)\leq \lambda'|w|$. Put $\lambda=\lambda'+(1-\lambda')/2<1$.

Let us construct the (multi-tape) Turing machine $M_x$ for $x\in X$ over the alphabet $S$ that, given a word $w\in S^{*}$, computes a word $M_x(w)\in S^{*}$ in linear time such that $M_x(w)=_Gw|_x$ and $|M_x(w)|\leq \lambda |w|$ for $|w|\geq L$. The machine $M_x$ operates as follows: given a word $w\in S^{*}$ of length $\geq L$, split $w$ into subwords of length $L$, $w=w^{(1)}w^{(2)}\ldots w^{(m)}v$, $|w^{(i)}|=L$ and $|v|<L$, and compute
\[
M_x(w) = w^{(1)}_{x_1}w^{(2)}_{x_2}\ldots w^{(m)}_{x_m} v|_{x_{m+1}}, \mbox{ where $x_1=x$ and $x_{i+1}=\pi_{w^{(i)}}(x_i)$.}
\]
Then $M_x$ works in linear time, $M_x(w)=_Gw|_x$ and $|M_x(w)|\leq\lambda'Lm+|v|\leq \lambda|w|$ for all words $w$ of length $\geq L$.

A multi-tape Turing machine $M$ solving the word problem in $G$ will have the input alphabet $\Gamma=S\cup\{\#\}$, the input tape, an $x$-tape for each $x\in X$, and tapes for operation of $M_x$. The machine $M$ will accept a given word $w=v_1\#v_2\# \ldots \#v_m$, $v_i\in S^{*}$, if $v_i=_Ge$ for every $i$, and rejects otherwise. The machine operates as follows.
\begin{enumerate}
  \item If $w$ is empty, then accept.
  \item Scan the input tape word $w=v_1\#v_2\# \ldots \#v_m$. If $|v_i|<L$ and $v_i\neq_Ge$, then reject. If $|v_i|<L$ and $v_i=_Ge$, then remove $v_i$ from $w$. 
  \item Scan the input tape word $w=v_1\#v_2\# \ldots \#v_m$. Compute the permutations $\pi_{v_i}$; if $\pi_{v_i}\neq e$ for some $i$, reject. 
  \item Scan the input tape word $w=v_1\#v_2\# \ldots \#v_m$. Compute $M_x(v_i)$ on the $x$-tape for each $x\in X$ simultaneously, for $i=1,2,\ldots,m$; if we read $\#$, then print $\#$ on each $x$-tape. After processing, the content of the $x$-tape is $M_x(v_1)\#M_x(v_2)\#\ldots\#M_x(v_m)$.
  \item Empty the input tape. Copy the content of the $x$-tapes into the input tape separating them with $\#$, and skip consecutive symbols $\#$ in the process. Empty the $x$-tapes. Go to step 1.
\end{enumerate}
At the $k$th stage of performing the steps 1-5 starting with a word $w\in S^{*}$, the input tape contains a word $w'=v_1\#v_2\#\ldots\#v_m$, where $v_i$ represent all the sections of $w$ at words of the $k$th level $X^k$. When we go from level to level, the length of the word on the input tape could increase; however, the length of each $v_i$ decreases exponentially, here $|v_i|\leq \lambda^k|w|$ at the $k$th stage. Therefore, the Turing machine $M$ stops after $O(\log n)$ stages.

Since $M_x$ work in linear time, the steps 1-5 are performed in linear time. For an input word $w$ of length $n$, the word size on the input tape at the $k$th stage is bounded by $2\beta^k n$, where $\beta=\lambda|X|$ and the two multiplier stands to count the maximal possible number of symbols $\#$ that could separate $\beta^k n$ symbols from $S$. Therefore, the time complexity of $M$ is bounded by
\begin{equation}\label{eqn:timecomp_T_item1}
O(n)+O(\beta n)+O( \beta^2 n)+\ldots O(\beta^{\log n} n)=Poly(n).
\end{equation}
The degree of the polynomial can be bounded by $\log_{\lambda^{-1}} |X|$ with $\lambda=1-\frac{1}{2L}$.

The item 3 goes the same way; the estimate~(\ref{eqn:timecomp_T_item1}) holds with $\beta<1$.

The item 2 is almost analogous. The estimate~(\ref{eqn:timecomp_T_item1}) holds with $\beta=1$, except that the number of stages of $M$ may not be bounded by $O(\log n)$; even worse, the computation $M(w)$ may even not stop in general, for example, when $w|_x=w$ and $\pi_w=\varepsilon$. Note that the condition in item 2 is preserved when we replace $X$ with $X^k$ for any $k\in\mathbb{N}$ or replace $L$ with a multiple $L'=mL$. We will show that one can choose $k$ and $L'$ so that item 1 holds; this will guarantee that $M$ stops in $O(\log n)$ stages.

If there exists $k\in\mathbb{N}$ such that $l_S(w|_x)<|w|$ for all $w\in S^{L}$ and $x\in X^k$, then item 1 holds for $X$ replaced with $X^k$, and we are done. Otherwise, let $N\subseteq S^{L}$ consist of words $w\in S^L$ such that, for every $n\in\mathbb{N}$, there exists $x_n\in X^n$ such that $l_S(w|_{x_n})=L$. Note that such an $x_n\in X^n$ is unique for a given $w\in N$ and $l_S(w|_y)=0$ for all $y\in X^n, y\neq x_n$. For every $w\in N$, by taking sections of $w$ at $x_n$ for $n=0,1,2,\ldots$, we get an eventually periodic sequence of words in $N$. Choose a large enough $k\in\mathbb{N}$ so that, by replacing $X$ with $X^k$, all periods in these sequences are equal to $1$ and all preperiods are equal to $0$ or $1$. Then, for every $w\in N$, there exist $x,y\in X$ such that $w|_x|_y=w|_x\in N$. In addition, we can guarantee that $l_S(w|_x)<L$ for all $w\in S^L\setminus N$ and $x\in X$.

Note that if two words $u,v\in S^{*}$ satisfy $\pi_u=\pi_v$ and $u|_x=u$, $v|_x=v$ for some $x\in X$ (then $u|_y=_Gv|_y=_Ge$ for $y\in X$, $y\neq x$), then $u=_Gv$. It follows that there are only finitely many group elements $F\subset G$ represented by words $w$ satisfying $w|_x=w$. Choose a large enough $m\in\mathbb{N}$ so that $l_S(g)<L'=mL$ for all $g\in F$.

Now consider a word $w\in S^{L'}$ and split $w=w_1w_2\ldots w_m$, $|w_i|=L$. If we assume $l_S(w|_x)=|w|$ for some $x\in X$, then $w_i\in N$ for all $i$. Then, either $l_S(w|_x|_y)<|w|$ for all $y\in X$ or $l_S(w|_x|_y)=|w|$ and $w|_x|_y=w|_x$ for some $y\in X$. The last case is not possible, because then $w|_x$ represents an element in $F$ and $l_S(w|_x)<L'$. Therefore, $l_S(w|_x)<|w|$ for all $w\in S^{L'}$ and $x\in X^2$, and we are done.
\end{proof}


\begin{corollary}\label{cor:strong_contr_virtfree}
Let $G$ be a strongly contracting group. Then $\WP_G\in\DTIME_{*}(n)$.
\end{corollary}

Note that all strongly contracting groups have subexponential growth (see \cite{Grigorchuk:Degrees}). Therefore, it is interesting whether $\WP_G\in\DTIME_1(o(n^2))$ for some strongly contracting group that is not virtually nilpotent.

\begin{example}
The Grigorchuk group is generated by the automaton over the alphabet $X=\{0,1\}$ with states $S=\{e,a,b,c,d\}$ and the transition map:
\begin{align*}
e\xrightarrow{0|0}e && a\xrightarrow{0|1}e && b\xrightarrow{0|0}a && c\xrightarrow{0|0}a && d\xrightarrow{0|0}e,\\
e\xrightarrow{1|1}e && a\xrightarrow{1|0}e && b\xrightarrow{1|1}c && c\xrightarrow{1|1}d && d\xrightarrow{1|1}b.
\end{align*}
The group $G$ is strongly contracting:
\begin{align*}
\forall w\in \{a,b,c,d\}^{10} \quad \sum_{x\in X^3} l_S(w|_x)<10.
\end{align*}
Hence, the word problem in $G$ is solvable in linear time by a multi-tape Turing machine. Such a Turing machine with $3$ tapes was constructed by my student Matei Chornomorets and is available at \cite{Chernom:TuringGri}.

\begin{corollary}
For the Grigorchuk group $G$, $\WP_G\in\DTIME_{*}(n)\subseteq\DTIME_1(n^2)$ and $\WP_G\not\in\DTIME_1(o(n^{1+\alpha}))$ for $\alpha=\frac{\log 2}{\log \eta}\approx 0.7674$, where $\eta$ is the positive root of the polynomial $x^3-x^2-2x-4$.
\end{corollary}
The branching property of the Grigorchuk group suggests that its word problem may not be solvable by a real-time Turing machine:

\begin{question}
Is the word problem in the Grigorchuk group solvable in real-time?
\end{question}

The Grigorchuk group $G$ is the first one in the family of Grigorchuk's groups $G_\chi$ parametrized by $\chi\in\{0,1,2\}^{\mathbb{N}}$. The complexity of the word problem in $G_\chi$ can be controlled by the complexity of the sequence $\chi$. This was used in \cite{Garzon:GriWP} to prove that the word problem in the groups $G_{\chi}$ separates the time complexity classes above $\DTIME_{*}(n^2)$.
\end{example}

\subsection{The word problem in groups of polynomial automata}

Recall that we assume that different states of generating automata produce different transformations.
The state of an automaton producing the identity transformation is called \textit{trivial} and denoted by $e$.

\begin{definition}
An automaton $A=(S,X,t,o)$ with the trivial state $e$ is called \textit{bounded} if there exists a constant $C$ such that for every $s\in S$ and $n\in\mathbb{N}$ there are at most $C$ words $v\in X^n$ such that $s|_v=e$.
\end{definition}

\begin{theorem}
Let $G$ be a group generated by a bounded automaton. Then $\WP_G\in \DTIME_{*}(n\log n)$. 
\end{theorem}
\begin{proof}
Groups generated by bounded automata are contracting (see \cite[Theorem~3.8.8]{Nekr:SSGroups} or \cite{BondNerk}). We can use the Turing machines $M$ and $M_x$ from the proof of Theorem~\ref{thm:timecomp_autom} with the following modification: after computing $M_x(w)$, remove every appearance of the letter $e$. The boundedness property implies that there exists a constant $C$ such that, for every $w\in S^{n}$ and $m\in\mathbb{N}$, the total length of words $w|_x$ for $x\in X^m$, after removing $e$, is bounded by $Cn$. Then, at every stage of computing $M(w)$ for $w\in S^n$, the length of the word on the input tape is bounded by $Cn$. Therefore, the running time of $M$ is bounded by (\ref{eqn:timecomp_T_item1}) with $\beta=1$. The result follows.
\end{proof}

\begin{example}
The Basilica group $B$ is generated by the bounded automaton over the alphabet $X=\{0,1\}$ with states $S=\{e,a,b\}$ and the transition map:
\begin{align*}
e\xrightarrow{0|0}e, \ e\xrightarrow{1|1}e && a\xrightarrow{0|1}e, \ a\xrightarrow{1|0}b && b\xrightarrow{0|0}e, \ b\xrightarrow{1|1}a.
\end{align*}
The group $B$ is not strongly contracting and has exponential growth; the semigroup $\langle a,b\rangle$ is free. The Turing machine $M$ from the proof above makes $\Omega(n2^n)$ steps on the word $(ab)^{2^n}$. Hence, the running time of $M$ is $\Theta(n\log n)$. I do not know whether the word problem in $B$ is solvable in linear time.
\end{example}

Bounded automata are polynomial automata of degree zero.

\begin{definition}
An automaton $A=(S,X,t,o)$ with the trivial state $e$ is called \textit{polynomial} if there exists a polynomial $P(n)$ such that for every $s\in S$ and $n\in\mathbb{N}$ there are at most $P(n)$ words $v\in X^n$ such that $s|_v\neq e$. The smallest degree of $P(n)$ with this property is called the \textit{degree} of $A$.
\end{definition}

Polynomial automata admit simple combinatorial characterization: different simple cycles at nontrivial states are disjoint. Groups generated by polynomial automata are not necessary contracting. Nevertheless, their word problem is effectively solvable.

\begin{theorem}
Let $G$ be a group generated by a polynomial automaton of degree $d$. Then $\WP_G\in \DTIME_{*}(n(\log n)^{(d+1)^2})$. If $G$ is contacting, then $\WP_G\in\DTIME_{*}(n(\log n)^{d+1})$.
\end{theorem}
\begin{proof}
Let us pass to a power of the alphabet so that every simple cycle in the generating automaton is a loop. We construct the Turing machine $M_x$ for computing a word $M_x(w)$ in linear time such that $M_x(w)=_Gw|_x$ as follows. For a word $w\in S^{n}$, compute the word $v=w|_x\in S^{n}$, and remove every appearance of the letter $e$ from $v$. Compare the words $v$ and $w$; if $v\neq w$, then return $v$; if $v=w$, then return the empty word.

Let $M$ be the Turing machine constructed in the proof of Theorem~\ref{thm:timecomp_autom} using the machines $M_x$. It follows from Lemma~1* from \cite{Bondarenko:Schreier} that the machine $M$ stops after $O( (\log n)^{d+1} )$ stages. The polynomial property implies that, at the $k$th stage of computing $M(w)$ for $w\in S^n$, the length of the word on the input tape is at most $nO(k^d)$. Therefore, the running time of $M$ is bounded by
\[
\sum_{k=1}^{ O((\log n)^{d+1}) } nO(k^d)=O( n(\log n)^{(d+1)^2}  ).
\]

If $G$ is contacting, we can use the Turing machines $M_x$ from the proof of Theorem~\ref{thm:timecomp_autom} in addition to removing the letter $e$. This will guarantee that $M$ performs $O(\log n)$ stages, and its running time is $O( n(\log n)^{d+1})$.
\end{proof}

\begin{example}
The smallest non-trivial example of a polynomial automaton of degree one is the automaton over
$X=\{0,1\}$ with states $S=\{e,a,b\}$ and the transition map:
\begin{align*}
e\xrightarrow{0|0}e, \ e\xrightarrow{1|1}e && a\xrightarrow{0|1}e, \ a\xrightarrow{1|0}a && b\xrightarrow{0|0}a, \ b\xrightarrow{1|1}b.
\end{align*}
The group $G$ generated by this automaton is not contracting.
The Turing machine $M$ from the proof above makes $\Omega(n^22^n)$ steps on the word $w=(baba^{-1})^{2^n}$. It follows that the running time of the Turing machine is $\Omega(n(\log n)^2)$.
I do not know whether the word problem in $G$ is solvable in linear time.
\end{example}

\begin{question}
Is the word problem in groups generated by bounded/polynomial automata solvable in linear time (by a multi-tape Turing machine)? If the answer is no, does the word problem in these groups distinguish the complexity classes $\DTIME_{*}(n(\log n)^d)$ for different $d\in\mathbb{N}$?
\end{question}

\end{document}